\documentclass[11pt]{amsart}
\usepackage{fullpage}
\usepackage{amsmath,amsthm}
\usepackage{comment}
\usepackage{amssymb}
\usepackage{graphicx,setspace,amscd,float,hyperref,enumitem}
\usepackage{xypic}

\newtheorem{thm}{Theorem}[section]
\newtheorem{lem}[thm]{Lemma}
\newtheorem{qst}[thm]{Question}
\newtheorem{prop}[thm]{Proposition}
\newtheorem{cor}[thm]{Corollary}
\theoremstyle{definition}
\newtheorem{df}[thm]{Definition}
\newtheorem{rk}[thm]{Remark}
\newtheorem{ex}[thm]{Example}

\newtheorem{con}[thm]{Convention}

\newtheorem*{mainthmA}{Theorem A}
\newtheorem*{mainthmB}{Theorem B}

\newcommand{\RR}{\mathbb R}

\newcommand{\G}{\Gamma}

\newcommand{\ol}{\overline}

\newcommand{\mA}{\mathcal{A}}

\newcommand{\vphi}{\varphi}
\newcommand{\veps}{\varepsilon}

  \newcommand{\Z}{\mathbb Z}

\newcommand{\from}{\colon}
\newcommand{\wt}{\widetilde}
\newcommand{\Lam}{\Lambda}

\newcommand{\out}{\textup{Out}(F_r)}
\newcommand{\stL}{\textup{Stab}(\Lambda^+_\vphi)}
\newcommand{\norm}{N(\langle\vphi\rangle)}
\newcommand{\norma}[1]{N(\c{\vphi^{#1}})}
\newcommand{\cent}{Cen(\langle\vphi\rangle)}
\newcommand{\com}{Comm(\langle\vphi\rangle)}

\newcommand{\os}{CV_r}

\newcommand{\aut}{\textup{Aut}(F_r)}

\renewcommand{\c}[1]{\langle #1 \rangle}
\newcommand{\ce}[2]{Cen_{#1}(\c{#2})}
\newcommand{\eps}{\epsilon}

\newcommand{\lam}{\lambda}

\newcommand{\gl}{\text{GL}(2,\Z)}
\newcommand{\pgl}{\mathbb{P}\text{GL}(2,\Z)}

\begin{document}

\title{Normalizers and centralizers of cyclic subgroups generated by lone axis fully irreducible outer automorphisms}
\author{Yael Algom-Kfir and Catherine Pfaff}

\address{\tt Department of Mathematics, University of Haifa \newline
  \indent Mount Carmel;  Haifa, 31905;  Israel
  \newline \indent  {\url{http://www.math.haifa.ac.il/algomkfir/}}, } \email{\tt yalgom@univ.haifa.ac.il}
  
\address{\tt Department of Mathematics, University of California at Santa Barbara \newline
  \indent South Hall, Room 6607; Santa Barbara, CA 93106-3080
  \newline 
  \indent  {\url{http://math.ucsb.edu/~cpfaff/}}, } \email{\tt cpfaff@math.ucsb.edu}

\date{}

\thanks{The first author is supported by ISF 1941/14. Both authors acknowledge support from U.S. National Science Foundation grants DMS 1107452, 1107263, 1107367 ``RNMS: Geometric structures And Representation varieties'' (the GEAR Network).} 

\maketitle

\begin{abstract}
We let $\vphi$ be an ageometric fully irreducible outer automorphism so that its Handel-Mosher \cite{hm11} axis bundle consists of a single unique axis (as in \cite{mp13}). We show that the centralizer $\cent$ of the cyclic subgroup generated by $\vphi$ equals the stabilizer $\stL$ of the attracting lamination $\Lambda^+_{\vphi}$ and is isomorphic to $\Z$. We further show, via an analogous result about the commensurator, that the normalizer $\norm$ of $\c{\vphi}$ is isomorphic to either $\Z$ or $\Z_2 * \Z_2$. 
\end{abstract}

\section{Introduction}

It is well known \cite{m94} that, given a pseudo-Anosov mapping class $\vphi$, the centralizer $\cent$ and normalizer $\norm$ of the cyclic subgroup $\langle \vphi \rangle$ 
are virtually cyclic. In fact, this property characterizes pseudo-Anosov mapping classes.\footnote{The justification of this fact is given by Sisto in \\ \url{http://mathoverflow.net/questions/82889/centralizers-of-non-iwip-elements-of-outf-n?rq=1} }

We recall some history surrounding this problem for the outer automorphism groups $\out$. 
In \cite{bfh97}, Bestvina, Feighn, and Handel constructed for a fully irreducible outer automorphism $\vphi \in \out$ the attracting lamination $\Lam^+_\vphi$. 
They proved that the stabilizer $\stL$ of $\Lam^+_\vphi$ in $\out$ is virtually cyclic (see also \cite[Thereom 4.4]{kl11}). 
Let $\com$ denote the commensurator of $\c{\vphi}$. 
Whenever the lamination $\Lam^+_\vphi$ is defined we have 
\[ \c{\vphi} \leq \cent \leq \stL  \leq \com \quad \text{and} \quad \cent \leq \norm  \leq \com.\]

In the fully irreducible case, the groups appearing above are all finite index subgroups of one another, and each of the inclusions may be strict (see Examples \ref{example1}, \ref{example2}, and \ref{example3}).

This article is concerned with identifying the centralizer and normalizer of $\c{\vphi}$ when $\vphi$ is an ageometric lone axis fully irreducible outer automorphism, as defined in Subsection \ref{ss:LoneAxisFullyIrreducibles}.  
Briefly, the term ``lone axis'' is used for when the axis bundle, defined by Handel and Mosher \cite{hm11}, consists of a single unique axis. The axis bundle is an analogue of the axis of a pseudo-Anosov, but in general contains many fold lines. 

\begin{mainthmA}\label{main1}
Let $\vphi \in \out$ be an ageometric fully irreducible outer automorphism such that the axis bundle $\mA_{\vphi}$ consists of a single unique axis, then $\cent = \stL \cong \Z$. 
\end{mainthmA}
 We then proved the following group theoretic corollary.

\begin{mainthmB}\label{main2}  
Let $\vphi \in \out$ be an ageometric fully irreducible outer automorphism such that the axis bundle $\mA_{\vphi}$ consists of a single unique axis, then either 
\begin{enumerate}
\item $\cent = \norm = \com \cong \Z$  or
\item $\cent \cong \Z$ and $\norm = \com \cong \Z_2 \ast \Z_2$.
\end{enumerate}
Further, in the second case, we have that $\vphi^{-1}$ is also an ageometric fully irreducible outer automorphism such that the axis bundle $\mA_{\vphi^{-1}}$ consists of a single unique axis. 
\end{mainthmB}

It is a consequence of \cite{IWGII} that ageometric lone axis fully irreducible outer automorphisms exist in each rank and it is proved in \cite{kp15} that this situation is generic along a specific ``train track directed'' random walk. It is noteworthy that the conditions for an outer automorphism to be an ageometric lone axis fully irreducible can be checked via the Coulbois computer package.\footnote{The Coulbois computer package is available at \cite{c12}.}

Understanding what properties transfer to inverses of outer automorphisms is in general elusive.
Theorem B gives a condition which guarantees that $\vphi^{-1}$ also admits a lone axis. 
However, we do not know if the latter case in fact occurs, prompting the following question.

\begin{qst} Does there exist some ageometric lone axis fully irreducible outer automorphism such that $\com = \norm \cong \Z_2 \ast \Z_2$? 
\end{qst}

We pose one further question.

\begin{qst} Can one give a concrete description of $\cent$ and $\norm$ when $\vphi$ is not an ageometric lone axis fully irreducible outer automorphism?
\end{qst}

Our strategy is to construct a homomorphism $\rho \from \stL \to (\RR,+)$ which is the signed translation distance along the axis (Lemma \ref{defOfRho}). We prove in Proposition \ref{psiIsId} that $\text{ker}(\rho)$ is trivial. Using Corollary \ref{centStl} we conclude Theorem A. Theorem B quickly follows.

Apart from the works of Bestiva, Feighn, and Handel \cite{bfh97} and Kapovich and Lustig \cite{kl11} mentioned above,  we  mention the following  results. 
Given any element $\vphi \in \out$, using the machinery of completely split relative train track maps, Feighn and Handel \cite{fh09} present an algorithm that virtually determines the weak centralizer of $\c{\vphi}$, i.e. all elements that commute with some power of $\vphi$. 
When $\vphi$ is a Dehn twist, Rodenhausen and Wade \cite{rw15} give an algorithm determining a presentation of $\cent$. They use this to compute a presentation of the centralizer of a Nielsen  generator.

\subsection*{Acknowledgements}
This paper came out of an idea presented to the second author by Koji Fujiwara after a talk she gave at Hebrew University. Both authors would like to thank Sam Ballas, Yuval Ginosar, Ilya Kapovich, Darren Long, Jon McCammond, and Lee Mosher for helpful and interesting conversations.

\section{Preliminary definitions and notation}{\label{ss:PrelimDfns}}

To keep this section at a reasonable length, we will provide only references for the definitions that are better known.

\subsection{Train track maps, Nielsen paths, and principal vertices}{\label{ss:TTsPNPs}}
Irreducible elements of $\out$ are defined in \cite{bh92} and fully irreducible outer automorphisms are those such that each of their powers is irreducible.
Every irreducible outer automorphism can be represented by a special kind of graph map called a train track map, as defined in \cite{bh92}. In particular, we will require that vertices map to vertices. 
Moreover, we can also choose these maps so that they are defined on graphs with no valence-1 or valence-2 vertices  (from the proof of \cite{bh92} Theorem 1.7). 
We refer the reader to \cite{bh92} for the definitions of \emph{directions}, \emph{periodic directions}, \emph{fixed directions},  \emph{legal paths}, \emph{Nielsen paths} ( \emph{NPs}) and \emph{periodic Nielsen paths} (\emph{PNPs}).

\begin{df}[Principal points] Given a train track map $g \colon  \Gamma \to \Gamma$, following \cite{hm11} we call a point \emph{principal} that is either the endpoint of a PNP or is a periodic vertex with $\geq 3$ periodic directions. Thus, in the absence of PNPs, a point is principal if and only if it is a periodic vertex with $\geq 3$ periodic directions
\end{df}

\subsection{Outer Space $\os$ and the attracting tree $T_+^{\varphi}$ for a fully irreducible $\vphi \in \out$}{\label{ss:OuterSpace}}
Let $\os$ denote the Culler-Vogtmann Outer Space in rank $r$, as defined in \cite{cv86}, with the asymmetric Lipschitz metric, as defined in \cite{ak_asymmetry}. 
 The group $\out$ acts naturally on $\os$ on the right by homeomorphisms. An element $\vphi \in \os$ sends a point $X = (\G, m, \ell) \in \os$ to the point $X \cdot \vphi = (\G, m \circ \Phi, \ell)$, where $\Phi$ is a lift in $\aut$ of $\vphi$. 
Let $\overline{\os}$ denote the compactification of $\os$, as defined in \cite{cl95, bf94}. The action of $\out$ on $\os$ extends to an action on $\ol{\os}$ by homeomorphisms. 

\begin{df}[Attracting tree $T_+^{\varphi}$]{\label{d:AttractingTree}}
Let $\vphi \in \out$ be a fully irreducible outer automorphism. Then $\vphi$ acts on $\ol{\os}$ with North-South dynamics (see \cite{ll03}). We denote by $T_+^{\varphi}$ the attracting fixed point of this action and by $T_-^{\varphi}$ the repelling fixed point of this action.
\end{df}

\vskip10pt

\subsection{The attracting lamination $\Lambda_{\varphi}$ for a fully irreducible outer automorphism.}{\label{ss:AttractingLamination}}

We give a concrete description of $\Lam^+_\vphi$ using a particular train track representative $g \from \G \to \G$. This is the original definition appearing in \cite{bfh97}. Note that apriori it is not clear that it does not depend on the train track representative.

\begin{df}[Iterating neighborhoods]{\label{d:IteratingNeighborhoods}}
Let $g \from \G \to \G$ be an affine irreducible train track map so that, in particular, there has been an identification of each edge $e$ of $\Gamma$ with an open interval of its length $\ell(e)$ determined by the Perron-Frobenius eigenvector. Let $\lam = \lam(\vphi)$ be its stretch factor and assume $\lam>1$. 
Let $x$ be a periodic point which is not a vertex (such points are dense in each edge). 
Let $\veps >0$ be sufficiently small so that the $\veps$-neighborhood of $x$, denoted $U$, is contained in the interior of an edge.  There exists an $N > 0$ such that $x$ is fixed, $U \subset g^N(U)$, and $Dg^N$ fixes the directions at $x$. We choose an isometry $\ell \colon (-\veps, \veps) \to U$ and extend it to the unique locally isometric immersion $\ell \colon \RR \to \Gamma$ so that $\ell(\lambda^Nt)=g^N(\ell(t))$. We then say that $\ell$ is \emph{obtained by iterating a neighborhood of $x$}.
\end{df}

\begin{df}[Leaf segments, equivalent isometric immersions]
We call isometric immersions $\gamma_1 \colon [a,b] \to \Gamma$ and $\gamma_2 \colon [c,d] \to \Gamma$ \emph{equivalent} when there exists an isometry $h \colon [a,b] \to [c,d]$ so that $\gamma_1 = \gamma_2 \circ h$. Let $\ell \from \RR \to \G$ be an isometric immersion. \emph{A leaf segment} of $\ell$ is the equivalence class of the restriction to a finite interval of $\RR$. Two isometric immersions $\ell$ and $\ell'$ are equivalent if each leaf segment of $\ell$  is a leaf segment of $\ell'$ and vice versa. 
\end{df}

\begin{df}[The realization in $\G$ of the attracting lamination $\Lambda^+_{\varphi}(\G)$]{\label{d:AttractingLamination}}
The \emph{attracting lamination realized in $\Gamma$}, denoted  $\Lambda^+_{\varphi}(\G)$, is the equivalence class of 
a line $\ell$ obtained by iterating a periodic point in $\G$ (as in Definition \ref{d:IteratingNeighborhoods}). An element of $\Lambda^+_\vphi(\Gamma)$ is called a leaf. 
Notice that $\Lambda^+_\vphi(\Gamma)$ can be realized as an $F_r$-invariant  set of bi-infinite geodesics in $\wt \G$, the universal cover of $\G$. We shall denote this set by $\Lambda^+_\vphi(\wt \Gamma)$.
\end{df}

The marking of $\G$ induces an identification of $\partial \G$ with $\partial F_r$. The attracting lamination $\Lam^+_\vphi$ is the image of $\Lam^+_\vphi(\wt \G)$ under this identification. In \cite{bfh97} it is proved that this set is independent of the choice of $g$.

\begin{df}[The action of $\out$ on the set of laminations $\Lambda^\pm_{\vphi}$]
Let $\psi \in \out$, then by \cite[Lemma 3.5]{bfh97},
\begin{equation}\label{eqLamAction}
\psi \cdot (\Lambda^+_{\vphi}, \Lambda^-_{\vphi}) = (\Lambda^+_{\psi\vphi\psi^{-1}}, \Lambda^-_{\psi\vphi\psi^{-1}}).
\end{equation}
\end{df}

\subsection{Whitehead graphs}{\label{ss:wg}}

The following definitions are in \cite{hm11} and \cite{mp13}.

\begin{df}[Stable Whitehead graphs and local Whitehead graphs]{\label{d:WhiteheadGraphs}}

Let $g \from \G \to \G$ be a train track map. The \emph{local Whitehead graph} 
$LW(v;\Gamma)$ at a point $v \in \Gamma$ has a vertex for each direction at $v$ and an edge connecting the vertices corresponding to the pair of directions $\{d_1,d_2\}$ if the turn $\{d_1,d_2\}$ is taken by an image of an edge.
The \emph{stable Whitehead graph} $SW(v;\Gamma)$ at a principal point $v$ is then the subgraph of $LW(v;\Gamma)$ obtained by restricting to the periodic direction vertices.
\end{df} 

The map $g$ induces a continuous simplicial map $Dg \from LW(g,v) \to LW(g,g(v))$. 
When $g$ is rotationless and $v$ a principal vertex, $Dg$ acts as the identity on $SW(g,v)$, when viewed as a subgraph of $LW(g,v)$, and hence gives an induced surjection $Dg \from LW(g,v) \to SW(g,v)$. 
We recall that for a train track representative of a fully irreducible outer automorphism the local Whitehead graph at each vertex is connected. Hence:

\begin{lem}
If $g \from \G \to \G$ is a train track map representing a fully irreducible outer automorphism $\vphi$ and $v \in \G$ is a principal vertex, then $SW(g,v)$ is connected. 
\end{lem}

\begin{lem}\label{tripodLemma}
Let $g \from \G \to \G$ be a rotationless PNP-free train track representative of an ageometric fully irreducible $\vphi \in \out$. Let $\wt \G$ be the universal cover of $\G$ and $\wt v \in \wt \G$ a vertex that projects to a principal vertex $v \in \G$. Then there exist two leaves $\ell_1, \ell_2$ of the lamination $\Lam_\vphi^+(\wt \G)$ so that $\ell_1 \cup \ell_2$ is a tripod whose vertex is $\wt v$. 
\end{lem}

\begin{proof}
Since $v$ is a principal vertex and there are no PNPs, we know $SW(g,v)$ has $\geq 3$ vertices. 
Since $SW(g,v)$ is connected, one of these vertices $d_1$ will belong to at least 2 edges $\eps_1, \eps_2$. Let $d_2,d_3$ be the directions corresponding to the other vertices of these edges. 
Since $g$ is rotationless, periodic directions are in fact fixed directions.
We may lift $g$ to a map $\wt g \from \wt \G \to \wt \G$ that fixes $\wt v$. 
Iterating the lifts of the edges that correspond to $d_1, d_2, d_3$ will give us three eigenrays $R_1, R_2, R_3$ initiating at $\wt v$. 
The 2 edges $\eps_1, \eps_2$ correspond to 2 leaves $\ell_1$ and $\ell_2$ of $\Lam_\vphi^+(\wt \G)$ \cite{hm11}. We have $\ell_1 \cup \ell_2 = R_1 \cup R_2 \cup R_3$. Hence, as desired, $\ell_1 \cup \ell_2$ is a tripod whose vertex is $\wt v$.
\end{proof}

\subsection{Axis bundles}{\label{ss:tt}}

Three equivalent definitions of the axis bundle $\mathcal{A}_{\varphi}$ for a nongeometric fully irreducible $\varphi \in \out$ are given in \cite{hm11}. We include only the definition that we use.

\begin{df}[Fold lines]\label{dfFoldLines}
A \emph{fold line} in $\os$ is a continuous, injective, proper function $\mathbb{R} \to \os$ defined by \newline
\noindent 1. a continuous 1-parameter family of marked graphs $t \to \Gamma_t$ and \newline
\noindent 2. a family of homotopy equivalences $h_{ts} \colon \Gamma_s \to \Gamma_t$ defined for $s \leq t \in \mathbb{R}$, each marking-preserving, \newline
\indent and satisfying:
~\\
\vspace{-\baselineskip}
\begin{description}
\item [\emph{Train track property}] $h_{ts}$ is a local isometry on each edge for all $s \leq t \in \mathbb{R}$. 
\item [\emph{Semiflow property}] $h_{ut} \circ h_{ts} = h_{us}$ for all $s \leq t \leq u \in \mathbb{R}$ and $h_{ss} \colon \Gamma_s \to \Gamma_s$ is the identity for all $s \in \mathbb{R}$.
\end{description}
\end{df}

\begin{df}[Axis Bundle]{\label{d:AxisBundle}} $\mathcal{A}_{\varphi}$ is the union of the images of all fold lines $\mathcal{F} \colon \mathbb{R} \to \os$ such that $\mathcal{F}$(t) converges in $\overline{\os}$ to $T_{-}^{\varphi}$ as $t \to -\infty$ and to $T_{+}^{\varphi}$ as $t \to +\infty$.
\end{df}

\begin{df}[Axes]{\label{d:Axes}}
We call the fold lines in Definition \ref{d:AxisBundle} the \emph{axes} of the axis bundle.
\end{df}

\subsection{Lone Axis Fully Irreducibles Outer Automorphisms}{\label{ss:LoneAxisFullyIrreducibles}}

\begin{df}[Lone axis fully irreducibles]
A fully irreducible $\vphi \in \out$ will be called a \emph{lone axis fully irreducible outer automorphism} if $\mathcal{A}_{\varphi}$ consists of a single unique axis.
\end{df}

\cite[Theorem 3.9]{mp13} gives necessary and sufficient conditions on an ageometric fully irreducible outer automorphism $\vphi \in \out$ to ensure that $\mathcal{A}_{\varphi}$ consists of a single unique axis. It is also proved there that, under these conditions, the axis will be the periodic fold line for a (in fact any) train track representative of $\vphi$. In particular, as is always true for axis bundles, $\mathcal{A}_{\varphi}$ contains each point in Outer Space on which there exists an affine train track representative of a power of $\vphi$.

\begin{rk}\label{noPNP}
It will be important for our purposes that no train track representative of an ageometric lone axis fully irreducible $\vphi$ has a periodic Nielsen path. This follows from \cite[Lemma 4.4]{mp13}, as it shows that each train track representative of each power of $\vphi$ is stable, hence (in the case of an ageometric fully irreducible outer automorphism) has no Nielsen paths.
\end{rk}

The following proposition is a direct consequence of \cite[Corollary 3.8]{mp13}. 

\begin{prop}[\cite{mp13}]\label{P:EveryVertexPrincipal}
Let $\vphi$ be an ageometric lone axis fully irreducible outer automorphism, then there exists a train track representative $g \from \Gamma \to \Gamma$ of some power $\vphi^R$ of $\vphi$ so that all vertices of $\Gamma$ are principal, and fixed, and all but one direction is fixed.
\end{prop}

\subsection{The stabilizer $\stL$ of the lamination}{\label{s:LaminationStabilizer}}

\begin{df}[$\stL$] Given a fully irreducible $\vphi \in \out$, we let $\stL$ denote the subgroup of $\out$ fixing $\Lambda^+_{\vphi}$ setwise, i.e. sending leaves of $\Lambda^+_{\vphi}$ to leaves of $\Lambda^+_{\vphi}$.
\end{df}

Bestvina, Feighn, and Handel \cite{bfh97} define a homomorphism (related to the expansion factor)  
\begin{equation}\label{homoDef}
\sigma \colon \stL \to (\RR_{> 0}, \cdot).
\end{equation}
They use $\sigma$ to prove the following theorem (\cite[Theorem 2.14]{bfh97}):

\begin{thm}[{\cite[Theorem 2.14]{bfh97}} or {\cite[Theorem 4.4]{kl11}}]{\label{t:stabilizerlamination}}
For each fully irreducible $\vphi \in \out$, we have that $\stL$ is virtually cyclic.
\end{thm}

\subsection{Commensurators}{\label{ss:Commensurators}}

Throughout this subsection let $G$ be a group and $H$ a subgroup of $G$.

\begin{df}[Commensurator $Comm_G(H)$]{\label{d:Sigma}}
The \emph{commensurator}, or \emph{virtual normalizer}, of $H$ in $G$ is defined as
$$Comm_G(H):=\{g \in G \mid [H:H \cap g^{-1}Hg] < \infty \text{ and } [g^{-1}Hg : H\cap g^{-1}Hg] < \infty \}.$$
Notice that when $H = \c{a}$, for some $a$, we have 
\begin{equation}\label{eqCom}
Comm_G(\c{a}):=\{g \in G \mid \exists \text{ }m,n \in \Z \text{ so that } g a^n g^{-1}= a^m \}. \end{equation}
\end{df}

\begin{rk}{\label{r:NormalizerSubCommensurator}}
$N_G(H) \leq Comm_G(H)$.
\end{rk}

\begin{prop}\label{cyclicImpliesEqual}
Let $a \in G$. If $Cen_G(\c{a}) \leq H$ and $H$ is cyclic, then $Cen_G(\c{a}) =H$.
\end{prop}

\begin{proof}
Since $H$ is cyclic, $H = \c{b}$ for some $b \in H$. Then $a = b^k$ for some $k$ and hence $b$ and $a$ commute. 
\end{proof}

\begin{prop}\label{ComEqNormalizer}
If $a \in G$ and $Comm_G(\c{a})$ is virtually cyclic, then for some $k \in \Z$ we have $N(\c{a^k}) = Comm_G(\c{a})$.
Thus, in this case, $N_G(\c{a}) \leq Comm_G(\c{a}) = N_G(\c{a^k})$. 
\end{prop}

\begin{proof}
First notice that $N_G(\c{a^k}) \leq Comm_G(\c{a^k}) \leq Comm_G(\c{a})$. Hence, we are left to show $Comm_G(\c{a}) \leq N_G(\c{a^k})$. Let $\c{a}$ have index $n$ in the group $Comm_G(\c{a})$. Let $b \in Comm_G(\c{a})$ and let $\omega_b \in Aut(G)$ denote conjugation by $b$. By (\ref{eqCom}), there exist $k,m \in \Z$ so that $ba^kb^{-1} = a^m$, hence $\omega_b(\c{a^k}) = \c{a^m}$. Now $n|m| = [Comm_G(\c{a}):\c{a^m}] = [Comm_G(\c{a}): \omega_b(\c{a^k})] = [\omega_b(Comm_G(\c{a})): \omega_b(\c{a^k})] = n|k|$. Hence, $|m|=|k|$ and so $b \in N_G(\c{a^k})$. 
\end{proof}

\section{The sequence of inclusions for a fully irreducible outer automorphism.}\label{s:Normalizer} 

\begin{con}[$\langle \vphi \rangle$, $\cent$, $\norm$] Given an element  $\vphi \in \out$, we let $\langle \vphi \rangle$ denote the cyclic subgroup generated by $\vphi$, we let $\cent$ denote its centralizer in $\out$, and we let $\norm$ denote its normalizer in $\out$.
\end{con}

\begin{lem}{\label{L:CommensuratorStabilizer}}
Let $\vphi \in Out(F_r)$ be fully irreducible.
Then: 
\begin{enumerate}
\item Each element $\psi \in \cent$ fixes the ordered pair $(T^+_{\vphi}, T^-_{\vphi})$ and the ordered pair $(\Lambda^+_{\vphi}, \Lambda^-_{\vphi})$. In particular, $\cent < \stL$.

\item $\com = Stab(\{\Lambda^+_{\vphi}, \Lambda^-_{\vphi}\}) = Stab(\{T^+_{\vphi}, T^-_{\vphi}\})$. 
And, in particular, each element $\psi \in \norm$ fixes the unordered pair $\{T^+_{\vphi}, T^-_{\vphi}\}$ and the unordered pair $\{\Lambda^+_{\vphi}, \Lambda^-_{\vphi}\}$.
\end{enumerate}
\end{lem}

\begin{proof}
(1) Let $\psi \in \cent$. Then by Equation \ref{eqLamAction} we have $\psi \cdot (\Lambda^+_{\vphi}, \Lambda^-_{\vphi}) = (\Lambda^+_{\vphi}, \Lambda^-_{\vphi})$.
That $\psi$ fixes the ordered pair $(T^+_{\vphi}, T^-_{\vphi})$ now follows from \cite[Lemma 3.5]{bfh97}.

(2) By Equations \ref{eqLamAction} and \ref{eqCom} we have 
$$\com \leq Stab(\{\Lambda^+_{\vphi}, \Lambda^-_{\vphi}\})$$
and by \cite[Lemma 3.5]{bfh97} this implies that 
$\com \leq Stab(\{T^+_{\vphi}, T^-_{\vphi}\})$. 
Now suppose $\psi \in Stab(\{\Lambda^+_{\vphi}, \Lambda^-_{\vphi}\})$. 
Then, by Equation \ref{eqLamAction} and \cite[Proposition 2.16]{bfh97}, we know $\psi\vphi^k\psi^{-1}=\vphi^n$ for some $k,n \in \Z$. So $\psi \in \com$.
\qedhere
\end{proof}

\begin{cor}\label{centStl}
If $\vphi \in \out$ is fully irreducible then there exists an integer $k \in \Z$ so that  
\[ \cent \leq \stL \leq \com = \norma{k}.\] 
Moreover, 
\begin{enumerate}
\item The subgroup in the right-hand inequality has index $\leq$ 2.
\item  If $\stL$ is cyclic then the left-hand inequality is an equality. 
\end{enumerate}
\end{cor}

\begin{proof}
\begin{enumerate}
\item $\stL = Stab(\Lam^+_\vphi, \Lam^-_\vphi)$. By Lemma \ref{L:CommensuratorStabilizer}(2), $\com = Stab(\{\Lam^+_\vphi, \Lam^-_\vphi \})$. Thus, $\vert \com : \stL \vert \leq 2$. By Proposition \ref{ComEqNormalizer}, there exists a $k \in \Z$ such that  $\com = \norma{k}$.
\item This follows from Proposition \ref{cyclicImpliesEqual}.
\end{enumerate}
\end{proof}

\begin{ex}\label{example1}
We work out an example where $\cent   = Cen(\langle\vphi^2\rangle) \lneq N(\langle\vphi^2\rangle) 
\leq \com$. Recall that $\text{Out}(F_2) \cong \gl$ via the abelianization map. 
Thus, it suffices to carry out the computations in $\gl$. Consider,
\[ A = \begin{pmatrix} 0 & 1\\1 & 1 \end{pmatrix}, 
\quad \quad  B = A^2 = \begin{pmatrix} 1 & 1\\1 & 2 \end{pmatrix},  \quad \quad P = \begin{pmatrix} 0 & 1 \\ -1 & 0 \end{pmatrix}.  
\]
The image $\bar{A}$ of $A$ in $\pgl$ acts on the hyperbolic plane by hyperbolic isometries fixing the points $\lam, -\frac{1}{\lam} \in \RR$. Each element of $\text{Stab}_{\pgl}(\lam, -\frac{1}{\lam})$ preserves the hyperbolic geodesic between $\lam$ and $-\frac{1}{\lam}$, we denote this by $[\lam,-\frac{1}{\lam}]$. 
Thus, the map $\rho \from \text{Stab}_{\pgl}(\lam, -\frac{1}{\lam}) \to (\RR,+)$ sending an element to its signed translation length on $[\lam,-\frac{1}{\lam}]$ is a homomorphism. Moreover, its image is discrete and its kernel is trivial. Hence, $\text{Stab}(\lam, -\frac{1}{\lam})$ is infinite cyclic and, by Proposition \ref{cyclicImpliesEqual}, $\text{Stab}(\lam, -\frac{1}{\lam})= Cen_{\pgl}(\c{\bar{A}})$. Since $\bar{A} \in \pgl$ is primitive then  $Cen_{\pgl}(\c{\overline{A}}) = \c{\bar{A}}$.  
One can check directly that $Cen_{\gl}(\c{A}) = \c{A,-I}$, where $I$ denotes the identity matrix (this follows since $\gl \to \pgl$ is 2-to-1). 
Similarly, $Cen_{\pgl}(\c{\bar{A}^2})$ is infinite cyclic and $\bar{A}$ is a primitive element of this group, hence  $Cen_{\pgl}(\c{\bar{A}^2}) =  \c{\bar{A}}$. Again, $Cen_{\gl}(\c{A^2}) = \c{A,-I}$. 
Moreover, one can check directly that $P \in N_{\gl}(\c{A^2})- N_{\gl}(\c{A})$. Hence, $Comm_{\gl}(\c{A}) \geq N_{\gl}(\c{A^2}) \geq \c{A,-I,P}$. 
\end{ex}

\begin{ex}{\label{example2}}
We show that there exists an ageometric fully irreducible outer automorphism $\vphi$ such that $\cent \ncong \Z$, and moreover $\cent \ncong \Z \times \Z_2$ (as in $\text{Out}(F_2)$, whose center is $\Z_2$). Consider $F_3 = \c{a,b,c}$. Let $R_3$ be the $3$-petaled rose and define
\[ \Psi: a \to b \to c \to ab. \]
It is straight-forward (see \cite[Proposition 4.1]{IWGII}) to check that this map represents an ageometric fully irreducible outer automorphism. 
Denote by $\Delta$ the 3-fold cover corresponding to the subgroup
\[ \c{b,c,a^3, abA, acA, a^2bA^2, a^2cA^2}.\]
We claim that $\Psi^{13}$ lifts to $\Delta$. Indeed, let $A$ be the transition matrix of $\Psi$, then 
\[ A^{13} = \begin{pmatrix}
7&9&12 \\
12&16&21\\
9&12&16
\end{pmatrix}.\]
In particular, both $\Psi(b)$ and $\Psi(c)$ cross $a$ a multiple of 3 times. Thus $\Psi^{13}$ lifts to $\Delta$. Denote the vertices of $\Delta$ by $v_0, v_1, v_2$. We denote by $g \from \Delta \to \Delta$ the lift of $\Psi^{13}$ that sends $v_0$ to itself. Let $T \from \Delta \to \Delta$ denote the deck transformation sending $v_0$ to $v_1$. The action of $T$ on $H_1(\Delta, \Z)$ is nontrivial, so $T$ does not represent an inner automorphism. 
Moreover, we claim that $g \circ T = T \circ g$. First note that the maps $g \circ T$ and $T \circ g$ are both lifts of $\Psi^{13}$. Moreover, since $a$ appears in $\Psi^{13}(a)$ 7 times (see the matrix $A^{13}$), $g(v_1) = v_1$. Therefore,
\[ g \circ T(v_0) = g(v_1) = v_1 = T(v_0) = T \circ g(v_0). \]
Thus, $g \circ T = T \circ g$. Let $\vphi \in \text{Out}(F_7)$ be the outer automorphism represented by $g$, and $\theta$ the outer automorphism represented by $T$. 
An elementary computation shows that $g$ is an irreducible train track map and that each local Whitehead graph is connected. Moreover, a PNP for $g$ would descend to a PNP for $\Psi$. Since $\Psi$ contains no such paths, there are no PNPs for $g$. Thus, the outer automorphism $\vphi$ is ageometric and fully irreducible (see \cite[Proposition 4.1]{IWGII}). In conclusion, $\theta$ is an order-3 element in $\ce{\text{Out}(F_7)}{\vphi}$, in contrast to the conclusion of our theorem for a lone axis ageometric fully irreducible outer automorphism. 
\end{ex}

\begin{ex}\label{example3}
In this example $\cent \lneq \stL$. 
Consider $\Psi$ as in Example \ref{example2} with its transition matrix $A$. We have:
\[ A^{16} = \begin{pmatrix}
16&21&28 \\
28&37&49\\
21&28&37
\end{pmatrix}\]
Thus, $\Psi^{16}$ lifts to a cover $\Delta$ corresponding  to  the index 7 subgroup of the free group 
\[\c{b,c,aba^{-1}, aca^{-1}, a^2ba^{-2}, a^2ca^{-2}, \dots ,a^6ba^{-6} ,a^6ca^{-6}, a^7}.\]
Number the vertices of $\Delta$ by $v_0, \dots v_6$. Let $g: \Delta \to \Delta$ be the lift of $\Psi^{16}$ fixing $v_0$. Since $\Psi^{16}(a)$ crosses $a$ a multiple of 16 times, $g(v_1) = v_2$. Thus, if $T$ is an order 7 deck transformation such that $T(v_i) = v_{i+1 \mod 7}$ then $g \circ T \neq T \circ g$. Let $\vphi$ denote the automorphism represented by $g$. Then as in the previous example, $\vphi$ is an ageometric fully irreducible outer automorphism. The lamination $\Lam^+_\vphi$ is a lift of the lamination $\Lam^+_\psi$ and therefore it is preserved by $T$. Thus $\theta$, the automorphism represented by $T$, is contained in $\stL$. But $\theta \notin \cent$. 
\end{ex}

\section{Proof of Main Theorems}{\label{s:Proof}}

\begin{lem}{\label{L:Fixers}}
Let $\vphi \in \out$ be an ageometric lone axis fully irreducible outer automorphism. If $\psi \in \out$ is an outer automorphism fixing the pair $(T^+_{\vphi}, T^-_{\vphi})$, then $\psi$ fixes $\mA_{\vphi}$ as a set, and also preserves its orientation.  
\end{lem}

\begin{proof} $\mA_{\vphi}$ consists precisely of all fold lines $\mathcal{F} \colon \mathbb{R} \to \os$ such that $\mathcal{F}(t)$ converges in $\overline{\os}$ to $T^{-}_{\varphi}$ as $t \to -\infty$ and to $T^{+}_{\varphi}$ as $t \to +\infty$. Further, since $\vphi \in \out$ is a lone axis fully irreducible outer automorphism, there is only one such fold line. Hence, since $\psi$ fixes $(T^+_{\vphi}, T^-_{\vphi})$, it suffices to show that the image of the single fold line $\mA_{\vphi}$ under $\psi$ is a fold line. 
Indeed given the fold line $t \to \G_t$ with the semi-flow family $\{h_{ts}\}$, the new fold line is just $t \to \G_t\cdot \psi$ with the same family of homotopy equivalences $\{h_{ts}\}$. Hence the properties of Definition \ref{dfFoldLines} still hold.  
\qedhere
\end{proof}

Recall that $\mA_\vphi$ is a directed geodesic and suppose that the map
$t \to \G_t$ is a parametrization of $\mA_\vphi$ according to arc-length with respect to the Lipschitz metric, i.e.
\begin{equation}\label{eqParametrization}
d(\G_t, \G_{t'}) = t'-t \quad \quad \text{for } t'>t.
\end{equation} 

\begin{lem}\label{defOfRho}
Let $\vphi \in \out$ be an ageometric lone axis fully irreducible outer automorphism and $\psi \in \stL$, then there exists a number $\rho(\psi) \in \RR$ so that for all $t \in \RR$, we have $\psi(\G_t) = \G_{\rho(\psi) + t}$. 
\end{lem}

\begin{proof}
$\stL = \text{Stab}(T^+_\vphi, T^-_\vphi)$ and by Lemma \ref{L:Fixers}, $\psi(\mA_\vphi) = \mA_\vphi$ and $\psi$ preserves the direction of the fold line. Therefore, there exists a strictly monotonically increasing surjective function $f \from \RR \to \RR$ so that $\psi(\G_t) = \G_{f(t)}$. Moreover, since $\psi$ is an isometry with respect to the Lipschitz metric, for $t< t'$, since $f(t)< f(t')$, Equation (\ref{eqParametrization}) implies
\[ f(t') - f(t) = d(\G_{f(t)}, \G_{f(t')}) = d(\psi(\G_{t}), \psi(\G_{t'})) = 
d(\G_t, \G_{t'}) = t'-t. \]
Hence $f(t') = f(t) + t' - t$. This implies that for all $s \in \RR$, 
$f(s) = f(0)+s$. Define $\rho(\psi) = f(0)$, then
\[ \psi(\G_t) = \G_{f(t)} = \G_{f(0)+t} = \G_{\rho(\psi) + t}. \qedhere \]  
\end{proof}

\begin{lem}{\label{L:Homomorphism}}
Let $\vphi \in \out$ be an ageometric lone axis fully irreducible outer automorphism, then the map $\rho \from \stL \to (\RR,+)$ is a homomorphism.
\end{lem}

\begin{proof} For each $t \in \RR$,
\[ \G_t = \psi^{-1} \psi(\G_t) = \psi^{-1}(\G_{\rho(\psi)+t}) =
\G_{\rho(\psi^{-1})+\rho(\psi) + t}.\]
Thus, $t = \rho(\psi^{-1}) + \rho(\psi) + t$, i.e. $\rho(\psi^{-1}) = - \rho(\psi)$. 
Moreover, let $\psi, \nu \in \stL$, then
\[ \G_{\rho(\psi\circ \nu)+t} = \psi \circ \nu(\G_t) = \psi(\nu(\G_t)) = \psi(\G_{\rho(\nu)+t}) = 
\G_{\rho(\psi) + \rho(\nu)+t}.\]
Thus, $\rho(\psi\circ \nu)= \rho(\psi) + \rho(\nu)$. 
We therefore obtain that $\rho$ is a homomorphism. 
\end{proof}

Since $\stL$ is virtually cyclic and $\rho(\vphi) \neq 0$, the image of $\stL$ under $\rho$ is infinite cyclic. Thus it gives rise to a surjective  homomorphism
\begin{equation}\label{eqHomo}
\tau \from \stL \to \Z
\end{equation} 
with a finite kernel. Note that the kernel consists precisely of those elements of $\out$ that, when acting on $\os$, fix the axis $\mA_\vphi$ pointwise. We show in Corollary \ref{kernelCor} that $ker(\tau)=id$.

\begin{prop}{\label{L:FixingAxis}}
Let $\vphi \in \out$ be an ageometric lone axis fully irreducible outer automorphism and let $\psi \in \stL$ be an outer automorphism that fixes $\mathcal{A}_\vphi$ pointwise. Let $f \from \G \to \G$ be an affine train track representative of some power $\vphi^R$ of $\vphi$ such that all vertices of $\G$ are principal and all directions but one are fixed (guaranteed by Proposition \ref{P:EveryVertexPrincipal}). Let $h\from \G \to \G$ be any isometry representing $\psi$.
Then $h$ permutes the $f$-fixed directions and hence fixes the (unique) nonfixed direction. 
\end{prop}

\begin{proof}
$\psi$ fixes the points $\G$ and $\G \vphi$. Thus there exist isometries $h \from \G \to \G$ and $h' \from \G \vphi \to \G \vphi$ that represent an automorphism $\Psi$ in the outer automorphism class of $\psi$, i.e. the following diagrams commute
\[\begin{array}{ll}
 \xymatrix{
R_r \ar[d]^{m} \ar[r]^{\Psi} &R_r \ar[d]^{m}\\
\G \ar[r]^{h} &\G} &
 \xymatrix{
R_r \ar[d]^{f \circ m} \ar[r]^{\Psi} &R_r \ar[d]^{f \circ m}\\
\G \ar[r]^{h'} &\G} 
\end{array}
\]
Therefore, the following diagram commutes up to homotopy
\[ \xymatrix{
\G \ar[d]^{f} \ar[r]^{h} &\G \ar[d]^{f }\\
\G \ar[r]^{h'} &\G} \]
We will show that this diagram commutes and in fact that $h'=h$. Let $H \from \G \times I \to \G$ be the homotopy so that $H(x,0) = f \circ h(x)$ and $H(x,1)= h' \circ f(x)$. Choose a  lift $\wt f$ of $f$ and a lift $\wt h$ of $h$ to $\wt \G$. Note that $\wt f \circ \wt h$ is a lift of $f \circ h$. 
Let $\wt H$ be a lift of $H$ that starts with the lift $\wt f \circ \wt h$. Then $\wt H(x,1)$ is a lift of $h' \circ f$, which we denote by $\wt{h' \circ f}$. This in turn determines a lift $\wt h'$ of $h'$ so that $\wt{h' \circ f} = \wt h' \circ \wt f$. 
There exists a constant $M$ so that for all $x \in \wt \G$, we have $d(\wt f \circ \wt h(x), \wt h' \circ \wt f(x))\leq M$, hence $\wt f \circ \wt h(x)$ and $\wt h' \circ \wt f(x)$ induce the same homeomorphism on $\partial \wt \G$.
Let $v \in \wt\G$ be any vertex. By Lemma \ref{tripodLemma}  
there exist leaves 
 $\ell_1, \ell_2$ of $\Lam_+(\wt \G)$ that form a tripod whose vertex is $v$. 
Then $\wt f \circ \wt h(\ell_1), \wt f \circ \wt h(\ell_2), \wt f \circ \wt h(\ell_3)$ are embedded lines forming a tripod, as are $\wt h' \circ \wt f(\ell_1), \wt h' \circ \wt f(\ell_2), \wt h' \circ \wt f(\ell_3)$. Moreover, the ends of the two tripods coincide. 
Thus, $\wt f \circ \wt h(v)= \wt h' \circ \wt f(v)$. Since $v$ was arbitrary and the maps are linear, we have $\wt f \circ \wt h = \wt h' \circ \wt f$ and $f \circ h = h' \circ f$. 

We now show that $h'=h$. Let $e_1$ be the oriented edge representing the nonfixed direction of $Df$. For all $i \neq 1$, $Df(e_i)= e_i$. Let $k$ be such that $h(e_k) = e_1$. We have $Dh' \circ Df = Df \circ Dh$. 
Thus for $i \neq 1, k$ we have $Dh'(e_i) = Dh(e_i)$. Since $h$ and $h'$ are isometries, this implies that $h'(e_i) = h(e_i)$ for $i \neq 1,k$. 
If $k=1$ then $h$ and $h'$ agree on all but one oriented edge and therefore coincide, so we assume $k \neq 1$. 
If $e_1 \neq \bar e_k$ then $h(\bar e_i) = h'(\bar e_i)$ for both $i=1$ and $i=k$, hence $h'=h$. Therefore we may assume that $\bar e_1 = e_k$. We have $h(e_k) = e_1$, hence $h(e_1) = e_k$. So $h'(\{e_1, e_k\})= \{e_1, e_k\}$, hence we assume $h'(e_k) = e_k$ and $h'(e_1) = e_1$. 
Notice that the edge of $e_1$ must be a loop, since $h$ and $h'$ coincide on all other edges. Further, the orientation of the loop is preserved by $h'$ and flipped by $h$.
Now let $j \neq 1$ be so that $Df(e_1)=e_j$ and let $u$ be an edge path so that $f(e_1)=e_j u e_1$. Thus, 
$f(e_k)= \overline{f(e_1)} = e_k \bar u \bar e_j$. 
We have 
\[ e_k \bar u \bar e_j= f(e_k) = f(h(e_1)) = h'(f(e_1))= h'(e_j) h'(u) h'(e_1). \]
Thus $h'(e_j)= e_k$, so $j=k$. Hence $Df(e_1)= e_k = Df(e_k)$. So the unique illegal turn of $f$ is $\{e_1, \bar e_1\}$. But this is impossible since $f$ is a homotopy equivalence and must fold to the identity. 
Thus, $h=h'$ and so, since we have from the previous paragraph that $f \circ h = h' \circ f$, we now know that the following diagram commutes 
\[ \xymatrix{
\G \ar[d]^{f} \ar[r]^{h} &\G \ar[d]^{f }\\
\G \ar[r]^{h} &\G} \]
Let $e$ be an edge so that the direction defined by $e$ is fixed by $Df$. We have $Dh(e) = Dh( Df(e)) = Df( Dh(e) )$, therefore $Dh(e)$ is also a fixed direction. Thus $h(e)$ defines a fixed direction, hence the $f$-fixed directions are permuted by $h$.
\end{proof}

\begin{prop}\label{psiIsId}
Under the conditions of Proposition \ref{L:FixingAxis}, $h$ is the identity on $\G$.
\end{prop}

\begin{proof}
Let $e$ be the oriented edge of $\G$ representing the unique direction that is not $f$-fixed (or $f$-periodic). By Proposition \ref{L:FixingAxis}, we know that $h(e) = e$. Let $p$ be an $f$-periodic point in the interior of $e$. We can switch to a power of $f$ fixing $p$. Let $\ell \in \Lam^+_\vphi(\G)$ be the leaf of the lamination obtained by iterating a neighborhood of $p$ (see Definition \ref{d:IteratingNeighborhoods}).
Denote by $\wt \G$ the universal cover of $\G$ and let $\wt p$ be a lift of $p$ and $\wt e$ and $\wt \ell$ be the corresponding lifts of $e$ and $\ell$. Let $\wt h$ and $\wt f$ be the respective lifts of $h$ and $f$ fixing the point $\wt p$. The lift $\wt f$ fixes $\wt \ell$, since this leaf is generated by $\wt f$-iterating a neighborhood of $\wt p$ contained in $\wt e$. 

We first claim $\wt f$ fixes only one leaf of $\wt{\Lambda}^+_\vphi(\G)$. Indeed, if $\wt \ell'$ is another such leaf, both ends of $\wt \ell'$ are $\wt f$-attracting, so there exists an $\wt f$-fixed point $\wt q \in \wt \ell'$. 
If $\wt q \neq \wt p$, then the segment between them is an NP, contradicting the fact that $f$ has no PNPs (see Remark \ref{noPNP}). Thus $\wt q = \wt p$. The intersection  $\wt \ell' \cap \wt \ell$ contains $\wt p$ but since $\wt p$ is not a branch point, it must also contain $\wt e$, i.e. the edge containing $\wt p$.  But since $\wt \ell$ and $\wt \ell'$ are both $\wt f$-fixed they must both contain $\wt f^k(\wt e)$ for each $k$. Thus $\wt \ell = \wt \ell'$.

We now claim that $\wt h(\wt \ell) = \wt \ell$. 
 By the previous paragraph, it suffices to show that $\wt f( \wt h(\wt \ell)) = \wt h(\wt \ell)$. 
We have $\wt h(\wt \ell) = \wt h \circ \wt f (\wt \ell) = \wt f \circ \wt h(\wt \ell) = \wt f( \wt h(\wt \ell))$, and our claim is proved. 

Recall from before that $\wt h(\wt e) =\wt e$. Since $\wt h$ is an isometry, it restricts to the identity on $\wt \ell$. Projecting to $\G$, since $\ell$ covers all of $\G$, we get that $h$ equals the identity on $\G$. 
\end{proof}

Recall the surjective homomorphism $\tau$ from Equation \ref{eqHomo}.

\begin{cor}{\label{kernelCor}} Let $\vphi \in \out$ be an ageometric fully irreducible outer automorphism such that the axis bundle $\mA_{\vphi}$ consists of a single unique axis, then $Ker(\tau)= \{id\}$.
\end{cor}

\begin{mainthmA}{\label{T:MainTheorem1}}
Let $\vphi \in \out$ be an ageometric fully irreducible outer automorphism such that the axis bundle $\mA_{\vphi}$ consists of a single unique axis, then $\cent = \stL \cong \Z$. 
\end{mainthmA}

\begin{proof} 
We showed in Corollary \ref{kernelCor} that $Ker(\tau)=id$. It then follows from Equation \ref{eqHomo} that $\stL \cong \Z$ the rest follows from Corollary \ref{centStl}.
\qedhere
\end{proof}

\begin{mainthmB}
Let $\vphi \in \out$ be an ageometric fully irreducible outer automorphism such that the axis bundle $\mA_{\vphi}$ consists of a single unique axis, then either 
\begin{enumerate}
\item $\cent = \norm = \com \cong \Z$  or
\item $\cent \cong \Z$ and $\norm = \com \cong \Z_2 \ast \Z_2$.
\end{enumerate}
Further, in the second case, we have that $\vphi^{-1}$ is also an ageometric fully irreducible outer automorphism such that the axis bundle $\mA_{\vphi^{-1}}$ consists of a single unique axis. 
\end{mainthmB}

\begin{proof} 
By Corollary \ref{centStl} and Theorem A, $\cent \leq \com$ and is of index $\leq 2$. If equality holds then we are in Case 1 and are done. Otherwise there is a short exact sequence
\begin{equation}\label{eqses} 
1 \to \cent   \to    \com  \to  \Z_2 \to 1. 
\end{equation}
There are two homomorphisms $\Z_2 \to \text{Aut}(\cent)$. We call the one whose image is the identity in $\text{Aut}(\cent)$ the trivial action and we call the one mapping the identity in $\Z_2$ to the automorphism in $\text{Aut}(\cent)$ taking a generator to its inverse the nontrivial action.
First suppose $\Z_2$ acts trivially. 
Let $\psi \in \com$ be any outer automorphism mapping to $1 \in \Z_2$, then $\psi \notin \cent$ and 
 $\psi \vphi \psi^{-1} = \vphi$ (because the action is trivial) and this is a contradiction. 
If $\Z_2$ acts nontrivially, then  $H^2(\Z_2, \Z) \cong \{0\}$ classifies the possible group extensions in the short exact sequence (\ref{eqses}) (see \cite[Proposition 3.7.3]{ben91}). Hence, the only possible extension is $\com \cong \cent \rtimes \Z_2 \cong \Z \rtimes \Z_2 \cong \Z_2 * \Z_2$.
Again let $\psi \notin \cent$ be any automorphism mapping to $1 \in \Z_2$. Since the homomorphism $\Z_2 \to \text{Aut}(\cent)$ is mapping 1 to the automorphism taking a generator to its inverse, $\psi \vphi \psi^{-1} = \vphi^{-1}$. Hence, $\psi \in \norm$.
So $\com = \c{\cent, \psi} \leq \norm$. On the other hand, by Remark \ref{r:NormalizerSubCommensurator}, $\norm \leq \com$. Hence, $\norm = \com$. 

We now prove the last part of the theorem. If $\com \cong \Z_2 * \Z_2$, then it contains an element $\psi$ mapping to the nonzero element in $\Z_2$ (as before) so that $\psi\vphi \psi^{-1}=\vphi^{-1}$. In other words, $\vphi^{-1}$ is in the conjugacy class of $\vphi$. Hence, it has the same index list and ideal Whitehead graph as $\vphi$ (and is also ageometric fully irreducible). In particular, $\vphi^{-1}$ satisfies the conditions to be a lone axis fully irreducible outer automorphism \cite[Theorem 4.6]{mp13}.
\qedhere
\end{proof}


\bibliographystyle{amsalpha}
\bibliography{PaperReferences}

\providecommand{\bysame}{\leavevmode\hbox to3em{\hrulefill}\thinspace}
\providecommand{\MR}{\relax\ifhmode\unskip\space\fi MR }
\providecommand{\MRhref}[2]{%
  \href{http://www.ams.org/mathscinet-getitem?mr=#1}{#2}
}
\providecommand{\href}[2]{#2}
\begin{thebibliography}{AKB12}

\bibitem[AKB12]{ak_asymmetry}
Yael Algom-Kfir and Mladen Bestvina, \emph{Asymmetry of outer space},
  Geometriae Dedicata \textbf{156} (2012), no.~1, 81--92.

\bibitem[Ben91]{ben91}
D.~J. Benson, \emph{Representations and cohomology. {I}}, Cambridge Studies in
  Advanced Mathematics, vol.~30, Cambridge University Press, Cambridge, 1991,
  Basic representation theory of finite groups and associative algebras.
  \MR{1110581 (92m:20005)}

\bibitem[BF12]{bf94}
Mladen Bestvina and Mark Feighn, \emph{{O}uter {L}imits},
  \url{http://andromeda.rutgers.edu/~feighn/papers/outer.pdf} (2012).

\bibitem[BFH97]{bfh97}
M.~Bestvina, M.~Feighn, and M.~Handel, \emph{Laminations, trees, and
  irreducible automorphisms of free groups}, Geometric and Functional Analysis
  \textbf{7} (1997), no.~2, 215--244.

\bibitem[BH92]{bh92}
M.~Bestvina and M.~Handel, \emph{Train tracks and automorphisms of free
  groups}, The Annals of Mathematics \textbf{135} (1992), no.~1, 1--51.

\bibitem[CL95]{cl95}
Marshall~M. Cohen and Martin Lustig, \emph{Very small group actions on {${\bf
  R}$}-trees and {D}ehn twist automorphisms}, Topology \textbf{34} (1995),
  no.~3, 575--617. \MR{1341810}

\bibitem[Cou14]{c12}
T.~Coulbois, \emph{{F}ree group automorphisms and train-track representative in
  python/sage}, \url{https://github.com/coulbois/sage-train-track}, 2012--2014.

\bibitem[CV86]{cv86}
M.~Culler and K.~Vogtmann, \emph{Moduli of graphs and automorphisms of free
  groups}, Inventiones mathematicae \textbf{84} (1986), no.~1, 91--119.

\bibitem[FH09]{fh09}
Mark Feighn and Michael Handel, \emph{Abelian subgroups of {${\rm Out}(F_n)$}},
  Geom. Topol. \textbf{13} (2009), no.~3, 1657--1727. \MR{2496054
  (2010h:20068)}

\bibitem[HM11]{hm11}
M.~Handel and L.~Mosher, \emph{Axes in {O}uter {S}pace}, no. 1004, Amer
  Mathematical Society, 2011.

\bibitem[KL11]{kl11}
I~Kapovich and M.~Lustig, \emph{Stabilizers of {R}-trees with free isometric
  actions of fn}, Journal of Group Theory \textbf{14} (2011), no.~5, 673--694.

\bibitem[KP15]{kp15}
I.~Kapovich and C.~Pfaff, \emph{A train track directed random walk on {${\rm
  Out}(F_r)$}}, International {J}ournal of {A}lgebra and {C}omputation
  \textbf{25} (August 2015), no.~5, 745--798.

\bibitem[LL03]{ll03}
G.~Levitt and M.~Lustig, \emph{Irreducible automorphisms of $ {F}_n$ have
  north--south dynamics on compactified outer space}, Journal of the Institute
  of Mathematics of Jussieu \textbf{2} (2003), no.~01, 59--72.

\bibitem[McC94]{m94}
J.~McCarthy, \emph{Normalizers and centralizers of pseudo-anosov mapping
  classes}, preprint), June \textbf{8} (1994).

\bibitem[MP13]{mp13}
L.~Mosher and C.~Pfaff, \emph{Lone axes in outer space}, arXiv preprint
  arXiv:1311.6855 (2013).

\bibitem[Pfa13]{IWGII}
C.~Pfaff, \emph{Ideal {W}hitehead graphs in {${\rm Out}(F_r)$} {II}: the
  complete graph in each rank}, Journal of Homotopy and Related Structures
  \textbf{10} (2013), no.~2, 275--301.

\bibitem[RW15]{rw15}
M.~Rodenhausen and R.~Wade, \emph{Centralisers of dehn twist automorphisms of
  free groups}, Mathematical Proceedings of the Cambridge Philosophical
  Society, vol. 159, Cambridge Univ Press, 2015, pp.~89--114.

\end{thebibliography}

\end{document}